\documentclass{eurocg11}

\newcommand{\comment}[1]{}

\usepackage{stmaryrd}
\usepackage{times}
\usepackage{wrapfig}
\usepackage{graphicx}
\usepackage{amsmath}
\usepackage{amsfonts}
\usepackage{amssymb}
\usepackage{graphics}
\usepackage{subfigure}

\title{Angular Convergence during B\'ezier Curve Approximation}

\author{J.Li\thanks{Department of Mathematics,
       University of Connecticut,Storrs, {\tt jili@math.uconn.edu}}
        \and
        T.J.Peters\thanks{Department of Computer Science and Engineering,
        Univeristy of Connecticut, Storrs, {\tt tpeters@cse.uconn.edu}}
        \and
        J.A.Roulier\thanks{Department of Computer Science and Engineering,
        Univeristy of Connecticut, Storrs, {\tt jrou@engr.uconn.edu}}}

\date{\today}

\begin{document}
\maketitle

\begin{abstract}
Properties of a parametric curve in $\mathbb{R}^3$ are often determined by analysis of its piecewise linear (PL) approximation. For B\'ezier curves, there are standard algorithms, known as subdivision, that  recursively create PL curves that converge to the curve in distance .  The {\em exterior angles} of PL curves under subdivision are shown to converge to $0$ at the rate of $O(\sqrt{\frac{1}{2^i}})$, where $i$ is the number of subdivisions. This angular convergence is useful for determining self-intersections and knot type.
\end{abstract}

\vspace{1ex}
\noindent
{\bf Keywords:} B\'ezier curve, subdivision, piecewise linear approximation, angular convergence.

\vspace{1ex}

\section{Introduction}
\label{sec:intro}

A B\'ezier curve is characterized by an indexed set of  points, which form a PL approximation of the curve\footnote{In the literature on B\'ezier curves, this PL approximation is called a {\em control polygon}.  However, to avoid confusion within this community, that terminology will be avoided here.}.  Subdivision algorithms recursively generate PL approximations that more closely approximate the curve under Hausdorff distance \cite{J.Munkres1999}.

Figure~\ref{fig:dec1} shows the first step of  a particular subdivision process, known as the de Casteljau algorithm.   The initial PL approximation $P$ is used as input to generate local PL approximations,  $P^1$ and $P^2$.  Their union, $P^1 \cup P^2$, is then a PL approximation whose Hausdorff distance is closer to the curve than that of $P$.  In this illustrative example, the construction relies upon generating midpoints of the segments of $P$, as indicated by providing $\frac{1}{2}$ as an input value to the subdivision algorithm\footnote{Other fractional values can be used, but the analysis given here proceeds by reliance upon $\frac{1}{2}$ and midpoints.  The details to change from midpoints are not substantive to the analysis presented here.}.

A brief overview is that subdivision proceeds by first creating the midpoint of each segment of $P$.  Then, these midpoints are connected to create new segments.
Recursive creation and connecting of midpoints continues until a final midpoint is created.  The union of the segments from the last step then forms a PL approximation.  Termination is guaranteed since $P$ has only finitely many segments.

\begin{figure}
\begin{center}\includegraphics[height=3cm]{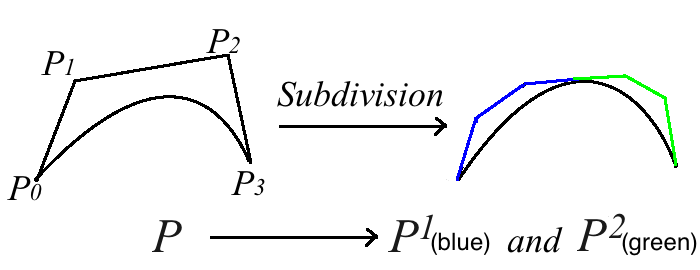}\end{center}
\caption{A subdivision}\label{fig:dec1}
\end{figure}

\begin{figure}
\begin{center}\includegraphics[height=3cm]{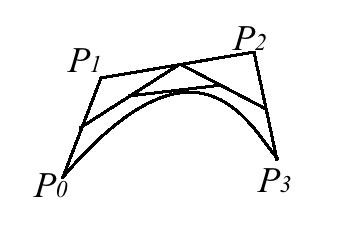}\end{center}
\caption{De Casteljau algorithm on $P$}\label{fig:dec2}
\end{figure}

This subdivision process is purely on PL geometry so that these techniques may be of interest to the computational geometry community.  
For subdivision, convergence of the PL curves to a B\'ezier curve under Hausdorff distance is well known \cite{PetersWu}, but, to the best of our knowledge, the convergence in terms of angular measure has not been previously established. The angular convergence is used to draw conclusions about non-self-intersection of the PL approximation\footnote{The authors thank J. Peters for conceptual insights offered during some informal conversatoins about the angular convergence. }.  Non-self-intersection is useful in determining more subtle topological properties, such as unknottedness and isotopic equivalence between a curve and this PL approximant \cite{Moore_Peters_Roulier2007}. This isotoptic equivalence has applications in computer graphics, computer animation and scientific visualization and the results presented here were discovered while
extending those previous theorems  \cite{Moore_Peters_Roulier2007}.  Further enhancements remain the subject of future research. 

For a simple\footnote{A curve is called {\em simple} if it is non-self-intersecting.} curve, it has been shown that the PL approximation under subdivision will eventually also become simple \cite{M.Neagu_E.Calcoen_B.Lacolle2000}. The proof relied upon use of the hodograph\footnote{The derivative of a B\'ezier curve is also expressed as a B\'ezier curve, known as the {\em hodograph} \cite{G.Farin1990}.} and is devoid of the constructive
geometric techniques used here.  That previous proof did not provide a specific rate of convergence, but the more geometric construction used here easily yields that convergence rate.

\section{Definitions and Notation}
Exterior angles were defined \cite{Milnor1950} for closed PL curves, but are adapted here for open curves.
\begin{defini}
\label{def:exterior_angles}
For an open PL curve with vertices $\{P_0,P_1,\ldots,P_n\}$ in $\mathbb{R}^3$,  denote the measures of the exterior angles formed by the oriented line segments to be:
\begin{center}
$\alpha_1, \ldots ,\alpha_{n-1}$ satisfying
\end{center}
\begin{center}
$0 \leq \alpha_m \leq \pi$ for $1 \leq m \leq n - 1.$ 
\end{center}
\end{defini}

For example, 
the exterior angle with measure $\alpha_m$ is formed by
$\overrightarrow{P_{m-1}P_m}$ and
$\overrightarrow{P_mP_{m+1}}$ and $0 \leq \alpha_m \leq \pi$, as shown in
Figure \ref{fig:alpha1}).  For these open PL curves, it is understood that the exterior angles are not defined at the initial and final vertices.

\begin{defini}
Denote $\mathcal{C}(t)$ as the parameterized B\'ezier curve of degree $n$ with control points $P_m \in\mathbb{R}^3$, defined by
$$
\mathcal{C}(t)=\sum_{m=0}^{n}{B_{m,n}(t)P_m},
t\in[0,1]
$$
where $B_{m,n}=\left(\!\!\!
  \begin{array}{c}
	n \\
	m
  \end{array}
  \!\!\!\right)t^m(1-t)^{n-m}$.
\end{defini}


Denote the uniform parametrization \cite{Morin_Goldman2001} of the PL curve $P$ by $l(P)_{[0,1]}$ over $[0,1]$, where $P=(P_0,P_1,\cdots,P_n)$. That is:
$$
l(P)_{[0,1]}(\frac{j}{n})=P_j\ for\ j=0,1,\cdots,n
$$
and $l(P)_{[0,1]}$ is piecewise linear.

\begin{defini}Discrete derivatives \cite{Morin_Goldman2001} are first defined at the points $t_j's$, where
$l(P)_{[0,1]}(t_j)=P_j$ for $j=0,1,\cdots,n-1$.
$$
P'_j=l'(P)_{[0,1]}(t_j)=\frac{P_{j+1}-P_j}{t_{j+1}-t_j}
$$
Denote $P'=(P'_0,P'_1,\cdots,P'_{n-1}).$ Then define the discrete derivative for $l(P)_{[0,1]}$ as:
$$
l'(P)_{[0,1]}=l(P')_{[0,1]}
$$

\end{defini}

Intuitively, the first discrete derivatives are similar to the tangent lines defined for univariate real-valued functions within a standard introductory calculus course.  


\begin{figure}
\begin{center}
\includegraphics[height=3cm]{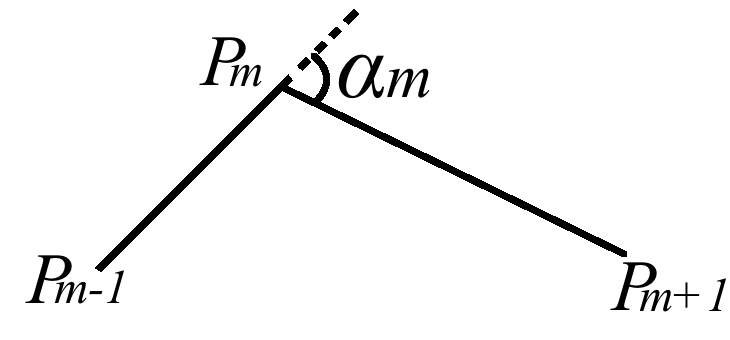}
\end{center}
\caption{The measure \textbf{$\alpha_m$} of an exterior angle}
\label{fig:alpha1}
\end{figure}

In order to avoid many annoying techincal considerations and to simplify the exposition, the class of B\'ezier curves considered will 
be restricted to those where the derivative never vanishes.

\begin{defini}
\label{def:regularity}
A differentiable curve is said to be regular if its derivative never vanishes.
\end{defini}

{\bf Throughout the rest of the presentation, the notation $\mathcal{C}$
will be used for a simple, $C^1$, regular\footnote{The astute reader will note that some of the development does
not require that the curve be regular.} B\'ezier curve of arbitrary degree $n$. And $i$ is the number of subdivisions. For convenience, the de Casteljau algorithm is assumed, with a fixed parameter $\frac{1}{2}$.}


\section{Angular Convergence}

For any B\'ezier curve, after $i$ iterations, the subdivision process generates $2^i$ open PL curves, whenever $i > 0$ (For $i = 0$, if the original PL curve formed from the control points is closed, then the associated B\'ezier curve is also closed.)  For the $i$-th subdivision, when $i > 0$ , denote each open PL approximation generated as
$$P^k=(P_0^k, P_1^k ,\ldots, P_n^k)$$ 
for $k = 1, 2, 3, \ldots, 2^i$. We consider an arbitrary $P^k$ for the following analysis, where, for simplicity of notation, we repress the superscript and denote this arbitrary curve simply as $P$, where $P$ has the corresponding parameters of the indicated control points by $t_0,t_1,\cdots,t_n$.  And let $l(P,i)$ be the uniform parameterization \cite{Morin_Goldman2001} of $P$ on  ${[\frac{k-1}{2^i},\frac{k}{2^i}]}$ $k \in \{ 1,2, 3, \ldots,2^i \}$. That is 
$$
l(P,i)=l(P)_{[\frac{k-1}{2^i},\frac{k}{2^i}]}\ \ \ and\ \ \  l(P,i)(t_m)=P_m
$$
Note from the domain of $l(P,i)$ that
\begin{align}
\label{eq:t} t_n-t_0=\frac{1}{2^i}
\end{align}
Furthermore, let
$$
\alpha_1,\alpha_2, \cdots, \alpha_{n-1}
$$
be the corresponding measures of
exterior angles of  $P$ (Definition
\ref{def:exterior_angles}).

Consider the Euclidian distance between the discrete derivatives of the two consecutive segments, that is $|l'(P,i)(t_m)-l'(P,i)(t_{m-1})|$, where $|\ |$ denotes Euclidean distance. We will show a rate of $O(\frac{1}{2^i})$ for the convergence
\begin{center}
$|l'(P,i)(t_m)-l'(P,i)(t_{m-1})| \rightarrow 0$ as $i \rightarrow \infty.$
\end{center} 
This will imply that $cos(\alpha_m) \rightarrow 1$ with the same rate and that  $\alpha_m \rightarrow 0$ at a rate of $O(\sqrt{\frac{1}{2^i}})$. Analogously we may imagine two connected segments in the x-y plane, and if their slopes are close, then their exterior angle is small. 

\begin{lemma}
\label{lem:derivative_diff}
For $\mathcal{C}$, the value $|l'(P,i)(t_m)-l'(P,i)(t_{m-1})|$ converges to zero at a rate of $O(\frac{1}{2^i})$. 
\end{lemma}
\begin{proof}
$$|l'(P,i)(t_m)-l'(P,i)(t_{m-1})|$$
$$\leq |l'(P,i)(t_m)-\mathcal{C}'(t_m)|+|\mathcal{C}'(t_m)-\mathcal{C}'(t_{m-1})|+$$
$$|\mathcal{C}'(t_{m-1})-l'(P,i)(t_{m-1})|$$
The first and the third terms
converge to $0$ at a rate of
$O(\frac{1}{2^i})$, because the discrete
derivative converges to the derivative of the original curve with this
rate \cite{Morin_Goldman2001}.

Now consider the convergence of the second term. Since $\mathcal{C}'$
satisfies the Lipschitz condition because of it being continuously
differentiable, we
have
$$|\mathcal{C}'(t_m)-\mathcal{C}'(t_{m-1})| \leq \gamma |t_m-t_{m-1}| \leq
\frac{\gamma}{2^i}$$
for some constant $\gamma$, where $\gamma$ does not depend on $t_m$ or $t_{m-1}$. The second inequality holds by the Equation (\ref{eq:t}).
Therefore $|l'(P,i)(t_m)-l'(P,i)(t_{m-1})|$ converges to zero
at a rate of $O(\frac{1}{2^i})$.
\end{proof}


\begin{theorem}[Angular convergence]\label{thm:curvature_conv}
For $\mathcal{C}$, each exterior angle of the PL curves generated by subdivision converges to $0$ at a rate of $O(\sqrt{\frac{1}{2^i}})$.
\end{theorem}

\begin{proof}
Since $\mathcal{C}(t)$ is assumed to be regular and $C^1$, the non-zero minimum of $|\mathcal{C}'(t)|$ over the compact set [0,1] is obtained. For brevity, the notations of $u_i  = l'(P,i)(t_m)$ and $v_i = l'(P,i)(t_{m-1})$ are introduced. The convergence of $u_i$ to $\mathcal{C}'(t_m)$ \cite{Morin_Goldman2001} implies that $|u_i|$ has a positive lower bound for $i$ sufficiently large, denoted by $\lambda$.

Lemma~\ref{lem:derivative_diff} gives that
$|u_i-v_i| \rightarrow 0$ as $i \rightarrow \infty$ at a rate of  $O(\frac{1}{2^i})$. This implies:
$|u_i|-|v_i| \rightarrow 0$ as $i \rightarrow \infty$ at a rate of  $O(\frac{1}{2^i})$.


Consider the following where the multiplication between vectors is dot product:
$$|cos(\alpha_m)-1|=|\frac{u_iv_i}{|u_i||v_i|}-1|$$
$$=|\frac{u_iv_i-v_iv_i+v_iv_i-|u_i||v_i|}{|u_i||v_i|}|$$
$$\leq \frac{|u_i-v_i|+||v_i|-|u_i||}{|u_i|} \leq \frac{|u_i-v_i|+||v_i|-|u_i||}{\lambda}$$
It follows from Lemma~\ref{lem:derivative_diff} that the right hand side converges to $0$ at a rate of  $O(\frac{1}{2^i})$. Consequently by the above inequality $|cos(\alpha_m)-1| \rightarrow 0$ with the same rate.

It follows from the continuity of $arccos$ that $\alpha_m$ converges to $0$ as $i \rightarrow \infty$.



Taking the power series expansion of $cos$ we get
$$|cos(\alpha_m)-1|
\geq (\alpha_m)^2 \cdot (\frac{1}{2}-|\frac{(\alpha_m)^2}{4!}-\frac{(\alpha_m)^4}{6!}+\cdots|)$$
Considering the expression on the right hand side of the previous inequality, note that for $1 > \alpha_m$, 
$$ e = 1+1+\frac{1}{2!}+\frac{1}{3!}+\frac{1}{4!}+\cdots > |\frac{1}{4!}-\frac{(\alpha_m)^2}{6!}+\cdots|.$$
For any $0< \tau <\frac{1}{2}$, sufficiently many subdivisions will guarantee that $\alpha_m$ is small enough such that
$1 > \alpha_m$ and $\tau > (\alpha_m)^2 \cdot e$. Then 
$$\tau >  (\alpha_m)^2 \cdot e > (\alpha_m)^2 \cdot |\frac{1}{4!}-\frac{(\alpha_m)^2}{6!}+\cdots| .$$
So $$|cos(\alpha_m)-1| > (\alpha_m)^2 \cdot (\frac{1}{2} - \tau)>0.$$
Then convergence of the left hand side implies that $\alpha_m$ converges to $0$ at a rate of $O(\sqrt{\frac{1}{2^i}})$.
\end{proof}

\section{Non-self-intersections from Subdivision}
\label{sec:simple-locally}

Even though the original PL approximant might not be simple, if the B\'ezier curve is simple, then subdivision eventually produces a PL approximant that is simple  \cite{M.Neagu_E.Calcoen_B.Lacolle2000}.  So, there must exist some value of $i$ such that, for the $i$-th subdivision each of the PL curves output as $(P_0^k, P_1^k ,\ldots, P_n^k)$ for $k = 1, 2, 3, \ldots, 2^i$ must also be simple.  Consistent with the approach taken here, that result will now be shown by a purely PL geometric construction which does not rely upon the hodograph of the B\'ezier curve.

Lemma~\ref{lem:non-int} is similar to one previously proven \cite{M.Neagu_E.Calcoen_B.Lacolle2000}, where the angles in the previous publication were defined over
a different range of values than used here.


The previous definition of exterior angles for open curves (Definition~\ref{def:exterior_angles}) was noted as a specialization of the original use for closed curves, where it was created to unify the concept of total curvature for closed curves that were PL or differentiable.   

\begin{defini}
The curvature of  $\mathcal{C}$ is given by 
$$\kappa(t)=\frac{||\mathcal{C}'(t) \times \mathcal{C}''(t)||}{||\mathcal{C}'(t)||^3}, \quad t \in [0,1]$$
Its total curvature \textup{\cite{DoCarmo1976}} is the integral: $\int_0^1 | \kappa(t)| \ dt.$
\end{defini}

\begin{figure}
\begin{center}
\includegraphics[height=4cm]{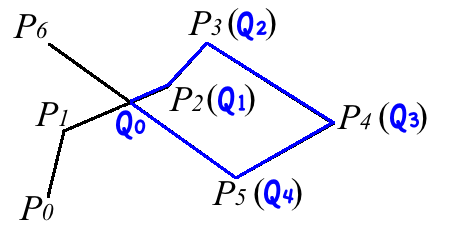}
\end{center}
\caption{A self-intersecting PL curve}
\label{fig:self-int}
\end{figure}

\begin{defini}\textup{\cite{Milnor1950}}
\label{def:dcurva}
For a closed PL curve $\bar{P}$ in $\mathbb{R}^3$, formed from points $P_0, P_1, \ldots, P_n$, its total curvature $\kappa(\bar{P})$ is defined as
$$
\kappa(\bar{P})=\sum_{m=0}^{n}{\alpha_m},.
$$
where $\alpha_0$ and $\alpha_m$ are both defined in the interval $[0, \pi]$, where $\alpha_0$ is formed from the edges $\overrightarrow{P_{n}P_0}$ and $\overrightarrow{P_{0}P_1}$, while  $\alpha_n$ is formed from the edges $\overrightarrow{P_{n-1}P_n}$ and $\overrightarrow{P_{n}P_0}$
\end{defini}

The following Fenchel's Theorem \cite{DoCarmo1976} is applicable both to PL curves and to differentiable curves.

\begin{theorem} \textup{\cite{Milnor1950}} \label{thm:gen-fenchel}
The total curvature of any closed curve is at least $2 \pi$, with equality holding if and only if the curve is convex.
\end{theorem}

\begin{lemma}\label{lem:non-int}
Let $P=(P_0,P_1,\cdots,P_n)$ be an open PL curve in $\mathbb{R}^3$.  If \hspace{1ex} $\sum_{j=1}^{n-1}{\alpha_j}<\pi$, then $P$ is simple.
\end{lemma}

\begin{proof}
Assume to the contrary that $P$ is self-intersecting.
Let $k$ be the lowest index for which the segment $P_{k-1}P_k$  intersects an earlier segment $P_{i-1}P_i$  for $i<k$. Consider the PL curve
$P_{i -1} \ldots P_k$ and isolate the single closed loop defined by the intersection as in Figure~\ref{fig:self-int}, where the designated intersection point is labeled as $Q_0$.  Denote this single closed loop by $\bar{Q}=(Q_0,Q_1,\cdots,Q_{n'},Q_0)$, for an appropriately chosen value of $n'$. Denote the measure of
the exterior angle of $\bar{Q}$ at $Q_0$ by $\beta_0$, where $\beta_0 \leq \pi$ (Definition~\ref{def:exterior_angles}).
Let $\kappa_l$ denote the total curvature of this closed loop, where $\kappa_l \geq 2 \pi$ (Theorem~\ref{thm:gen-fenchel}). Then, faithfully index
the remaining angles by an oriented traverse of $\bar{Q}$ such that each exterior angle has, respectively, measure $\beta_j$ for $j = 1, \ldots, n'.$
  Note that for $j \geq 1$, each $\beta_j$ is equal to some $\alpha_i$.

Since $\bar{Q}$ is a portion of $P$, we have
$$\sum_{i=1}^{n-1}{\alpha_j}  > \sum_{j=1}^{n'}{\beta_j}.$$
Note also that $\sum_{j=0}^{n'}{\beta_j}\geq 2 \pi$ (Theorem~\ref{thm:gen-fenchel}). But, then since $\beta_0 \leq \pi$, it follows that $\sum_{j=1}^{n'}{\beta_j} \geq \pi$ and $\sum_{i=1}^{n-1}{\alpha_j}\geq \pi$, which is a contradiction.

Two subtleties of the proof are worth mentioning.  First, the assumption that each angle $\alpha_i < \pi$ precludes the case of two
consecutive edges intersecting at more than a single point. 
Secondly, the choice of $k$ as the lowest index is sufficiently general to include the case where two non-consecutive edges are coincident.
\end{proof}

\pagebreak

\begin{theorem}\label{thm:non-self-int-sub}
For $\mathcal{C}$, there exists a sufficiently large value of $i$, such that after $i$-many subdivisions, each of the PL curves generated as $P^k=(P_0^k, P_1^k ,\ldots, P_n^k)$ for
$k = 1, 2, 3, \ldots, 2^i$ will be simple.
\end{theorem}


\begin{proof}
Since the measure of each exterior angle converges to zero as $i$ increases (Theorem~\ref{thm:curvature_conv})
and since each open $P^k=(P_0^k, P_1^k ,\ldots, P_n^k)$ has $n - 1$ edges, there exists $i$ sufficiently large such that
$$ \sum_{j = 1}^{n - 1} \alpha_j^k  < \pi,$$ for each $k = 1, 2, 3, \ldots, 2^i.$
Use of Lemma~\ref{lem:non-int} completes the proof.
\end{proof}

\section{Conclusions and Future Work}


Total curvature is fundamental for determining knot type, as applicable to both PL curves and differentiable curves.

\begin{theorem}
\label{Fary-Milnor_thm}
Fary-Milnor Theorem \textup{\cite{Milnor1950}}:
If the total curvature of a
simple closed curve is less than or equal to
$4\pi$, then it is unknotted.
\end{theorem}

These local topological properties of the PL approximation for a B\'ezier curve have been instrumental to showing isotopic equivalence between $\mathcal{C}$ and the polyline approximation generated by subdivision for low degree B\'ezier curves. Efforts are ongoing to extend that isotopic equivalence to higher degree B\'ezier curves.  


\small
\bibliographystyle{abbrv}
\bibliography{ji-tjp-biblio}

\end{document}